\documentclass[10pt,a4paper]{amsart}
\usepackage[pdftex]{graphicx} 
\usepackage[usenames,dvipsnames]{xcolor}
\usepackage{hhline}
\usepackage[english]{babel} 
\usepackage[centering,margin=2.5cm]{geometry}
\usepackage{upgreek}
\usepackage[font=small,labelfont=bf]{caption}
\usepackage{amsfonts}
\usepackage{tensor} 
\usepackage{lmodern}
\usepackage{verbatim}
\usepackage{amsmath}
\usepackage[all, cmtip]{xy}\usepackage{amssymb}
\usepackage{amsthm}
\usepackage{amsbsy}
\usepackage{extpfeil}
\usepackage{bbm}
\usepackage{paralist}
\usepackage[colorlinks,citecolor=magenta,pagebackref=true,urlcolor=magenta,
pdftex]{hyperref}
\usepackage{cancel}

\usepackage{soul}

\usepackage[T1]{fontenc}

\usepackage{nicefrac}
\usepackage{microtype}
\usepackage{mathrsfs}
\usepackage[nodisplayskipstretch]{setspace}
\newtheorem{theorem}{\bf Theorem}[section]

\newtheorem{cor}[theorem]{Corollary}
\newtheorem{lemma}[theorem]{Lemma}
\newtheorem{proposition}[theorem]{Proposition}

\newtheorem{thmnonumber}{\bf Theorem}
\newtheorem{cornonumber}{\bf Corollary} 

\theoremstyle{definition}
\newtheorem{rem}[theorem]{Remark}
\newtheorem{definition}[theorem]{Definition}

\newcommand{\St}{\mr{st}}
\newcommand{\Lk}{\mr{lk}}

\newcommand{\cm}[1]{}
\newcommand\mc[1]{\mathcal{#1}}

\newcommand\mr[1]{\mathrm{#1}}

\newcommand{\R}{\mathbb{R}}

\newcommand{\CAT}{\mathrm{CAT}}
\newcommand\F{\mc{F}}

\newcommand{\LLk}{\mathrm{LLk}^{\,\nu} }

\newcommand{\intx}{\mathrm{int}\,}
\newcommand{\sd}{\mathrm{sd}\, }
\newcommand{\RS}{\mathrm{R}\,}
\newcommand{\TT}{\mathrm{T}}
\newcommand{\rint}{\mathrm{relint}}
\newcommand{\RN}{\mathrm{N}}

\title{Polyhedral CAT(0) metrics on locally finite complexes}
\author{Karim A.~Adiprasito}
\address{
Jussieu Institute of Mathematics - Paris Rive Gauche, 
4 place Jussieu,
Mailbox 247, 
75252 Paris Cedex 5, France}
\email{
karim.adiprasito@imj-prg.fr}

\author{Louis Funar}
\address{Univ. Grenoble Alpes, CNRS, Institut Fourier, 38000 Grenoble, France}
\email{louis.funar@univ-grenoble-alpes.fr}

\date{\today}
\thanks{}

\keywords{}
\subjclass[2010]{57N16, 51K10, 57N15.}

\begin{document}

\begin{abstract}
We prove the arborescence of any locally finite complex that is $\CAT(0)$ with a polyhedral metric for which all vertex stars are convex. In particular locally finite $\CAT(0)$ cube complexes or equilateral simplicial complexes are arborescent. Moreover, a 
triangulated manifold admits a $\CAT(0)$ polyhedral metric if and only if it admits arborescent triangulations. We prove eventually that every locally finite complex which is 
$\CAT(0)$ with a polyhedral metric has a barycentric subdivision which is arborescent. 
\end{abstract}

\maketitle 

\section{Introduction}

$\CAT(0)$ spaces are metric spaces for which, roughly speaking,  geodesic triangles are thinner than the Euclidean triangles with the same edge lengths (see \cite{GromovHG}). The Hadamard--Cartan theorem states that simply connected, complete, locally $\CAT(0)$ spaces are (globally) $\CAT(0)$ and hence contractible. The topology of contractible manifolds admitting $\CAT(0)$ metrics is known to be constrained, according to \cite{Lytchak,Rolfsen,Stone}. Exotic examples of such manifolds have been constructed by Davis and Januszkiewicz (see \cite{DavisJanus}) and later Ancel and Guilbault  (\cite{AG2}) showed that interiors of all compact contractible manifolds admit $\CAT(0)$ metrics. In  our previous work \cite{AF} we aimed at giving an almost complete characterization of their  topology under the assumption that the $\CAT(0)$-metric is polyhedral, e.g. piecewise Euclidean.  
If one looks more generally upon complexes instead of manifolds White (see \cite{Wh}) proved that a simplicial 2-complex admits a strongly convex metric if and only if it is collapsible. As such the result cannot be extended to higher dimensions, as there exist (non rectilinear) triangulations of the 3-cell which are not collapsible. 
Further Crowley (see \cite{Crowley}) proved that  3-dimensional $\CAT(0)$ complexes satisfying an additional condition including triangulated 3-manifolds whose piecewise Euclidean metric is $\CAT(0)$  are collapsible after subdivision. Although Crowley only considered finite 3-complexes, the infinite case can be proved the same way. 
Finite $\CAT(0)$ cubical complexes were known to be collapsible for a while, see   
(\cite{CCMW}, section 4) for more general results in this direction. Moreover, in  \cite{ABpart1} the authors proved that all finite $\CAT(0)$ complexes are collapsible after subdivision.

Our main goal here is to understand which {\em locally finite} contractible 
complexes admit polyhedral $\CAT(0)$ metrics, in particular to settle the description from \cite{AF} for manifolds. 

A key topological notion in this article is the
{\em  arborescence} of an infinite cell complex, which is the natural extension of the 
{\em collapsibility} to the realm of non-compact spaces.  Specifically, a locally finite complex is \emph{arborescent} if it is obtained 
from a vertex by infinitely many dilatations, or equivalently, if it is the ascending union of 
finite subcomplexes, each finite subcomplex collapsing on the smaller one.

In \cite{BL} the authors considered a weaker notion of arborescence of a locally finite complex which only asks for being the ascending union of collapsible complexes. It seems that this weak arborescence does not imply the arborescence, in general. In \cite{AF} we provided examples 
of locally finite complexes which are ascending unions of compact contractible submanifolds but which are not semi-stable at infinity and hence are not arborescent. On the other hand each contractible submanifold has collapsible triangulations, by \cite{ABpart1}.  
As locally finite $\CAT(0)$ complexes are ascending unions of convex finite subcomplexes and hence $\CAT(0)$ subcomplexes, they are automatically weakly arborescent.

Recall that a  $\CAT(\kappa)$ metric on a polytopal complex  is said to be {\em polyhedral} 
if every cell, when equipped with the induced metric,  
is isometric to the convex hull of a finite set of points in the spherical, hyperbolic or Euclidean 
space of curvature $\kappa$. For instance the {\em piecewise
 flat equilateral metric}   is the length metric obtained when 
simplices or cubes are Euclidean and have all their edges of the same (unit) length. 

Our first result is:

\begin{thmnonumber}[Theorem~\ref{thm:dischadamard}]
\label{mainthm:CATcollapsible}
Let $C$ be a locally finite simplicial complex that is $\CAT(0)$ with a polyhedral metric for which all vertex stars are convex. Then $C$ is arborescent.
\end{thmnonumber}

{\em Convexity of vertex stars} means that for each vertex $v$, with respect to the metric introduced on $C$, the segment between any two points of the star $\St (v,C)$ lies in $\St (v,C)$. It is automatically satisfied by any complex in which all simplices are acute or right-angled.

 As a consequence of the first result we 
obtain:

\begin{cornonumber}\label{mainthm:CATcube}
Every locally finite complex that is $\CAT(0)$ with the equilateral flat metric is arborescent. Furthermore, every locally finite $\CAT(0)$ cube complex is arborescent.
\end{cornonumber}

The case of $\CAT(0)$-cube complexes has been proved independently by Gulbrandsen in his PhD thesis (see \cite{Gul}).  Note also that the corner peeling method used in \cite{CC} provides an alternative proof.

This is a converse of the result proved earlier in \cite{AF} by the authors stating that 
an arborescent locally finite simplicial complex is PL homeomorphic to a locally finite $\CAT(0)$ cubical complex.  
In particular we have the following characterization of arborescence: 

\begin{cornonumber}\label{mainthm:CAT(0)man}
A simplicial complex admits an arborescent triangulation if and only if it is PL homeomorphic to a $\CAT(0)$ cube complex. In particular, for any integer $d$, the open $d$-manifolds admitting an arborescent triangulation are precisely those that admit a $\CAT(0)$ polyhedral metric. 
\end{cornonumber}

Throughout this paper all polyhedral metrics are supposed to have only {\em finitely many} isometry classes 
of cells. 

We only consider locally finite complexes endowed with polyhedral 
$\CAT(\kappa)$ metrics which are {\em complete geodesic} metric spaces. 
Explicit examples can be obtained from a locally finite triangulation of manifolds 
whose simplices are endowed with constant curvature metrics and have finite distorsion, by (\cite{BH}, I.7A.13 and I.3.7).

Note however that all contractible manifolds do not admit arborescent 
triangulations, as there are specific restrictions on the topology at infinity 
(see \cite{AF}). Nevertheless the interiors of compact contractible manifolds admit CAT(0) metrics, see \cite{AS, AG2}. 

The usual definition of polyhedral metrics from  (\cite{BH}, I.7, Def. 7.37) allows the above condition be satisfied 
for a suitable subdivision of the  polytopal complex. However, for our next result it is important to work with 
a fixed geometric cell decomposition. The second main result gives a complete characterization of complexes admitting a polyhedral $\CAT(0)$ metric, as follows: 

\begin{thmnonumber}[Theorem \ref{arborescent}]\label{mainthm:CATpolyhedral}
Every locally finite complex that is $\CAT(0)$ with a polyhedral metric has an arborescent barycentric subdivision.
\end{thmnonumber}
This improves the second part of  Theorem 1 from \cite{AF}.   
The proofs build upon the Forman discrete Morse theory (\cite{FormanUSER}) and its 
developments leading to the collapsibility results of \cite{ABpart1,ABpart2}.   

{\bf Acknowledgements.} The authors are grateful to Victor Chepoi, Craig Guilbault and 
Daniel Gulbrandsen for useful discussions. The first author is supported by the ISF, the CNRS and the Horizon Europe ERC  HodgeGeoComb Grant no. 101045750.

\section{Preliminaries} \label{sec:Preliminaries}

\subsection{Definitions}
A \emph{polytope} is the convex hull of finitely many points in a simply connected space of constant curvature. 
When the curvature is positive we ask the points to belong to some open hemisphere-sphere. 
A \emph{polytopal complex} is a finite collection of polytopes in the ambient space such that the intersection of any two polytopes is a face of both. An \emph{intrinsic polytopal complex} is a collection of polytopes that are attached along isometries of their faces so that the intersection of any two polytopes is a face of both. 

The \emph{underlying space} $|C|$ of a polytopal complex $C$ is the topological space obtained by taking the union of its faces. A \emph{triangulation} 
of a topological space $X$ is a simplicial complex with a homeomorphism $|C|\to X$. Two polytopal complexes $C,\, D$ are \emph{equivalent} if their face posets are isomorphic, in which case  their underlying spaces are homeomorphic. 

A \emph{subdivision} of a polytopal complex $C$ is a polytopal complex $C'$ with the same underlying space of $C$, such that for every face $F'$ of $C'$ there is some face $F$ of $C$ for which $F' \subset F$. Two polytopal complexes $C$ and $D$ are \emph{PL equivalent} if 
they have equivalent subdivisions.  A triangulation of a topological manifold is \emph{PL}  if the star of every face  is PL equivalent to the simplex of the same dimension.

A subdivision  $\sd  C$ of a polytopal complex $C$ obtained by stellarly subdividing all its faces in decreasing  order of their dimensions is called \emph{derived subdivision}, see \cite{Hudson}. An example of a derived subdivision is the barycentric subdivision, which uses as vertices the barycenters of all faces of $C$.

If $C$ is a polytopal complex and $A$ is a subset, the \emph{restriction $\RS(C,A)$ of $C$ to $A$} is the inclusion-maximal subcomplex $D$ of $C$ such that $D$ lies in $A$. The \emph{star} of $\sigma$ in $C$, denoted by $\St(\sigma, C)$, is the minimal subcomplex of $C$ that contains all faces of $C$ containing $\sigma$. The \emph{deletion} $C-D$ of a subcomplex $D$ from $C$ is the subcomplex of $C$ given by $\RS(C,  C{\setminus} \rint{D})$.
The \emph{(first) derived neighborhood} $N(D,C)$ of $D$ in $C$ is the simplicial complex
\[N(D,C):=\bigcup_{\sigma\in \sd D} \St(\sigma,\sd C). \]

\subsection{Polyhedral CAT(k) spaces and convex subsets} 
Recall that a metric space $(X,d)$ is  {\em geodesic} (also called a length space or an inner metric space) if every two points 
of it can be joined by a minimizing geodesic, namely a curve whose length equals the distance between the points. 
The length of the continuous path $\gamma:[0,1]\to X$ is defined as 
\[ \sup_{r, 0=t_0 <t_1 <\cdots t_r<t_{r+1}=1}\; \; \; \;  \sum_{j=0}^{r} d(\gamma(t_j),\gamma(t_{j+1}) \]
A geodesic triangle in $(X,d)$ satisfies the CAT($\kappa$) inequality if the geodesic comparison triangle 
with sides of the same length within  the simply connected curvature $\kappa$ Riemannian surface has 
distances between pairs of boundary points larger than those between corresponding pairs of points in the initial triangle. 
Moreover, the geodesic metric space $(X,d)$ is CAT(0) if  every geodesic triangle, which for $\kappa >0$ has perimeter less than $\frac{2\pi}{\sqrt{\kappa}}$, satisfies the CAT($\kappa$) inequality.

Given two points $a,b$ in a length space $X$, we denote sometimes by $|ab|$ the \emph{distance} between $a$ and $b$, which is  the minimum of the lengths of all curves from $a$ to $b$.
In a $\CAT(0)$ space, any two points are connected by a unique geodesic. The same holds for $\CAT(k)$ spaces ($k >0$), as long as the two points are at distance $ < \pi k^{-\frac{1}{2}}$.

Let $c$ be a point of $X$ and let $K$ be a closed subset of $X$, not necessarily convex. We denote by $\pi_c(K)$ the subset of the points of $K$ at minimum distance from $c$, the {\em closest-point projection} of $c$ to $K$. In case $\pi_c(K)$ contains a single point, with abuse of notation we write $\pi_c(K) = x$ instead of $\pi_c(K) = \{x\}$. This is always the case when $K$ is convex, as the following well-known lemma shows. 

\begin{lemma}[{\cite[Prop.\ 2.4]{BH}}]\label{lem:unique}
Let $X$ be a connected $\CAT(k)$-space, $k\le 0$.  Let $c$ be a point of $X$. Then the function ``distance from $c$'' has a unique local minimum on each closed convex subset $K$ of $X$. Similarly, if $k>0$ and $K$ is at distance less than $\tfrac{1}{2}\pi k^{\nicefrac{-1}{2}}$ from $c$, then there exists a unique local minimum of distance at most $\tfrac{1}{2}\pi k^{\nicefrac{-1}{2}}$ from $c$.
\end{lemma}  

Classical results show that every open contractible $m$-manifold $M$ is triangulable, namely there exists a {\em locally finite simplicial complex} $\Delta$ homeomorphic to $M$ (see e.g. \cite{KS}, Annex B, p. 300, Annex C, p.315). 
The CAT($\kappa$) metrics which we consider on $M$ are supposed to be {\em polyhedral}, namely there exists a suitable 
triangulation $\Delta$ such that every cell of $\Delta$, when equipped with the induced metric,  
is isometric to the convex hull of a finite set of points in the hyperbolic or Euclidean 
space of curvature $\kappa\leq 0$, or within a hemisphere when $\kappa> 0$ (see \cite{BH}, I.7, Def. 7.37). For instance the {\em piecewise flat equilateral metric}   is the length metric obtained when 
simplices or cubes are Euclidean and have all their edges of the same (unit) length. 

Throughout this paper all polyhedral metrics are supposed to have only {\em finitely many} isometry classes 
of cells. 

In \cite{AF} we pointed out that the results presented there were valid more generally  for $\CAT(0)$ 
metrics for which the restriction to every cell of $\Delta$ is a  
piecewise analytic Riemannian metric whose curvature is bounded above and below by two constants independent on the 
cell and that the distorsion of cells is uniformly bounded. The polyhedral requirements in this article seem 
to be necessary for most results presented here.

Note that we only consider topological manifolds endowed with polyhedral 
CAT($\kappa$) metrics which are {\em complete geodesic} metric spaces. 
Explicit examples can be obtained from a locally finite triangulation of $M$ 
whose simplices are endowed with constant curvature metrics and have finite distorsion, by (\cite{BH}, I.7A.13 and I.3.7).

\subsection{Tangent cones, geodesics and links}
We define the notion of \emph{link} with a metric approach. However we will restrict ourselves to the case where 
the space $X$ is a polytopal complex endowed with a $\CAT(\kappa)$ polyhedral metric. In particular every cell $\sigma$ 
of $X$ is realized as a totally geodesic convex polytope in the space form of constant curvature $\kappa$.  

Let $p$ be any point of $X$. By $\TT_p X$ we denote the {\em tangent cone} of $X$ at $p$, namely the union $\bigcup_{p\in \sigma} T_p\sigma$ of all tangent spaces of cells $\sigma$ containing $p$, where we identify $T_p\tau$ and its image 
within $T_p\sigma$, for any cell $\tau\subset \sigma$.   
Each tangent space $T_p\sigma$ inherits a positive bilinear form 
from the constant curvature Riemannian metric on $\sigma$. The various scalar products are compatible on intersections since 
cell embeddings are totally geodesic.  

Note that two points $p,q$ in the same cell $\sigma$ determine an unique geodesic  $\gamma$ joining them. 
This is clear for $\kappa\leq 0$, and implied by our assumptions that spherical cells be embedded into hemispheres, if $\kappa >0$. Moreover, as cells are totally geodesic submanifolds of the corresponding space form 
the geodesic $\gamma$ is smooth, actually real analytic, and the tangent vector is well-defined at any of its points. 
Further, given a point $p$ of $X$ and a vector $v\in T_pX$, there exists an unique geodesic $\gamma$ 
issued from $p$ whose tangent vector is $v$. 
    
This allows us to define the tangent cone $T_pF$ to a point $p$ of a convex set $F\subset X$. 
Assuming $p\in \partial F$, as otherwise it is clear, we define 
$T_pY$ to be the cone of those vectors $v\in T_pY$ for which there exists $q\in F$, in a given neighborhood of $q$ within $X$, with the property that the geodesic $\gamma$ joining $p$ to $q$ has tangent vector $v$ at $p$ and lives entirely in a 
cell of $X$. When $\partial F$ is a polyhedron, $T_pY$ is a closed subcone of $T_pY$.

Let $\TT^1_p X$ denote the set of unit vectors in the cone $\TT_p X$.  If $Y$ is any subspace of $X$ and $p\in Y$ such that 
the tangent cone $T_pY$ makes sense, then $\RN_{(p,Y)} X$ denotes the cone of $\TT_p X$ spanned by the vectors orthogonal to $\TT_p Y$. If $p$ is in the interior of $Y$, we define $\RN^1_{(p,Y)} X:= \RN_{(p,Y)} X \cap \TT^1_p Y$. 

This situation occurs, for instance,  when 
$Y$ is locally a finite union of manifolds, a convex subspace or the boundary of a convex domain. 
Consider  a convex subset $Y$ of a the underlying space $X$ of a polytopal complex endowed with a polyhedral 
 $\CAT(\kappa)$ metric and a point $p$ which belongs to its boundary $\partial Y$.
Note that  $Y$ is then a topological manifold with boundary. Being convex means that at every point $Y$ has 
totally geodesic support hyperplanes. Recall that a totally geodesic hyperplane $H$ contained in a cell $\sigma$ is 
a support hyperplane at $p$ for the convex set $Y\subset \sigma$ if it contains $p$ and $Y$ is contained in the closure 
of one connected component $H_+$ of $\sigma\setminus H$.   The tangent cone $T_pY$ consist of the intersection of all components 
$H_+$ containing $Y$, where $H$ denotes the totally geodesic support hyperplanes at $p$. 
Moreover, $\RN_{(p,\partial Y)} X= \bigcap _{H}\RN_{(p,H)} X$, where $H$ belongs to the set of totally geodesic hyperplane supports at $p$.

If $\tau$ is any face of a polytopal complex $C$ containing a nonempty face $\sigma$ of $C$, then the set $\RN^1_{(p,\sigma)} \tau$ of unit tangent vectors in $\RN^1_{(p,\sigma)} |C|$ pointing towards $\tau$ forms a spherical polytope $P_p(\tau)$, isometrically embedded in $\RN^1_{(p,\sigma)} |C|$. The family of all polytopes $P_p(\tau)$ in $\RN^1_{(p,\sigma)} |C|$ obtained for all $\tau \supset \sigma$ forms a polytopal complex, called the \emph{link} of $C$ at $\sigma$ and denoted by $\Lk_p(\sigma, C)$. If $C$ is a polytopal complex endowed with a polyhedral $\CAT(\kappa)$ metric, then $\Lk_p(\sigma, C)$ is naturally realized in $\RN^1_{(p,\sigma)} X$. When $X$ is a PL $d$-manifold $\RN^1_{(p,\sigma)} X$ is isometric to a sphere of dimension $d-\dim \sigma -1$, and will be considered as such. Up to ambient isometry $\Lk_p(\sigma, C)$ and  $\RN^1_{(p,\sigma)} \tau$ in $ \RN^1_{(p,\sigma)} |C|$ or $\RN^1_{(p,\sigma)} X$ do not depend on $p$.

If $C$ is simplicial, and $v$ is a vertex of $C$, then $\Lk(v,C)$ is combinatorially equivalent to \[(C-v)\cap \St(v,C)=\St(v,C)-v.\]By convention, $\Lk(\emptyset, C):=C$. If $C$ is a simplicial complex, and $\sigma$, $\tau$ are faces of $C$, then $\sigma\ast \tau$ is the minimal face of $C$ containing both $\sigma$ and $\tau$ (assuming it exists). If $\sigma$ is a face of $C$, and $\tau$ is a face of $\Lk(\sigma,C)$, then $\sigma \ast \tau$ is the face of $C$ with $\Lk(\sigma,\sigma \ast \tau)=\tau$. In both cases, the operation~$\ast$ is called the \emph{join}.

\subsection{Discrete Morse theory  after Forman and arborescence}

The \emph{face poset} $(C, \subseteq)$ of a polytopal complex $C$ is the set of nonempty faces of $C$, ordered with respect to inclusion. By $(\mathbb{R}, \le)$ we denote the poset of real numbers with the usual ordering. A \emph{discrete Morse function} is an order-preserving map $f$ from $(C, \subseteq)$ to $(\mathbb{R}, \le)$, such that the preimage $f^{-1}(r)$ of any number $r$ consists of either one element, or of two elements one contained in the other. A \emph{critical cell} of $C$ is a face at which $f$ is strictly increasing.  

The function $f$ induces a perfect matching on the non-critical cells: two cells are matched whenever they have identical image under $f$. This is called \emph{Morse matching}, and it is usually represented by a system of arrows: Whenever $\sigma \subsetneq \tau$ and $f(\sigma) = f(\tau)$, one draws an arrow from the barycenter of $\sigma$ to the barycenter of $\tau$. We consider two discrete Morse functions  \emph{equivalent} if they induce the same Morse matching. Since any Morse matching pairs together faces of different dimensions, we can always represent a Morse matching by its  associated partial function  
$\Theta$ from $ C$ to itself, defined on a subset 
of $C$ as follows:

\[
\Theta (\sigma) \ := \ 
\left\{
\begin{array}{cl}
\sigma & \textrm{if $\sigma$ is unmatched}, \\
\tau & \textrm{if $\sigma$ is matched with $\tau$ and $\dim \sigma < \dim \tau$}.
\end{array}
\right.
\]

A {\it discrete vector field} $V$ on a polytopal complex $C$ is a collection of pairs $(\sigma,\Sigma)$ of faces such that $\sigma$ is a codimension-one face of $\Sigma$, and no face of $C$ belongs to two different pairs of $V$. 
A \emph{gradient path} in $V$ is a concatenation of pairs of $V$
\[ (\sigma_0, \Sigma_0), (\sigma_1, \Sigma_1),  \ldots, (\sigma_k, \Sigma_k),\, k\ge 1,\]
so that for each $i$ the face $\sigma_{i+1}$ is a codimension-one face of $\Sigma_i$ different from $\sigma_i$. A gradient path is \emph{closed} if $\sigma_0 = \sigma_k$ for some $k$ (that is, if the gradient path forms a closed loop). A discrete vector field $V$ is a {\em Morse matching}  if $V$ contains no closed gradient paths \cite{FormanADV,FormanUSER}.

Inside a polytopal complex $C$, a \emph{free} face $\sigma$ is a face strictly contained in only one other face of $C$. An \emph{elementary collapse} is the deletion of a free face $\sigma$ from a polytopal complex~$C$. We say that $C$ \emph{(elementarily) collapses} onto $C-\sigma$, and write $C\searrow_e C-\sigma.$ In the opposite direction an inverse elementary collapse is called a dilatation and we write $C-\sigma\nearrow_e C$. 
We also say that the complex $C$ \emph{collapses} to a subcomplex $C'$, and write~$C\searrow C'$, if $C$ can be reduced to $C'$ by a sequence of elementary collapses. A \emph{collapsible} complex is a complex that collapses onto a single vertex. 

Moreover, a locally finite complex is \emph{arborescent} if it is obtained 
from a vertex by infinitely many dilatations. Alternatively, there exists an ascending union by 
finite subcomplexes $K_i$, such that $K_{i+1}$ collapses onto $K_i$, for every $i$. 

Collapsibility only depends on the combinatorial type  and does not depend on the geometric realization of a polytopal complex. Discrete Morse theory provides a simple criterion for collapsibility, as follows:

\begin{theorem}[Forman \cite{FormanADV}]\label{lem:equivalence}
A finite polytopal complex $C$ is collapsible if and only if $C$ admits a discrete Morse function with only one critical face.
\end{theorem}

The following is an immediate extension of Forman's result to infinite complexes: 

\begin{theorem}\label{lem:equivalenceinf}
A locally finite polytopal complex $C$ is arborescent if $C$ admits a discrete Morse function with only one critical face.
\end{theorem}

Collapsible complexes are contractible and manifolds which have PL triangulations which are collapsible are necessarily PL balls \cite{Whitehead}. In the same way arborescent complexes are contractible.  We will prove later that an open manifold which has an arborescent PL triangulation is PL homeomorphic to the Euclidean space (see Proposition \ref{openwhitehead}). 

Note that there exist in every dimension $d\geq 4$ compact contractible manifolds which have non simply connected boundary and hence they are 
not homeomorphic to the ball. Every such contractible manifold admits  some 
collapsible non PL triangulation. 
  
On the other hand  an open contractible manifold has not necessarily a  
(even non PL) triangulation which is arborescent. When it does, it is PL homeomorphic to a cubical complex (see \cite{AF}) and in particular its topology at infinity is quite restricted.

We now state the following lemmas from \cite{ABpart1}, for further use: 

\begin{lemma}\label{lem:ccoll}
Let $C$ be a simplicial complex, and let $C'$ be a subcomplex of $C$.  Then the cone over base $C$ collapses to the cone over $C'$.
\end{lemma}

\begin{lemma}\label{lem:cecoll}
Let $v$ be any vertex of any simplicial complex $C$. If $\Lk(v,C)$ collapses to some subcomplex $S$, then $C$ collapses to 
$(C-v)\cup (v\ast S).$ In particular, if $\Lk(v,C)$ is collapsible, then  $C\searrow C-v$.  
\end{lemma}

\begin{lemma}\label{lem:uc}
Let $C$ denote a simplicial complex that collapses to a subcomplex $C'$. Let $D$ be a simplicial complex such that $D\cup C$ is a simplicial complex. If $D\cap C=C'$, then $D\cup C\searrow D$.
\end{lemma}

\subsection{Gradient matchings and star-minimal functions following \cite{ABpart1,ABpart2}}\label{subsec:1}
In this section we show how to obtain Morse matchings on a simplicial complex, using real-valued continuous functions on the complex.

\begin{definition}
Let $C$ be an intrinsic simplicial complex. A function $f : |C| \rightarrow \mathbb{R}$ is called \emph{star-minimal} if it satisfies the following three conditions:
\begin{compactenum}[(i)]
\item $f$ is continuous,
\item on the star of each face of $C$, the function $f$ has a unique absolute minimum, and
\item no two vertices have the same value under $f$.
\end{compactenum}
\end{definition}

Note that a generic continuous function on a simplicial complex is star-minimal.

A star-minimal function on a complex $C$ induces a Morse matching, called \emph{gradient matching}, as follows. 

\begin{figure}[ht]
	\centering
  \includegraphics[width=.35\linewidth]{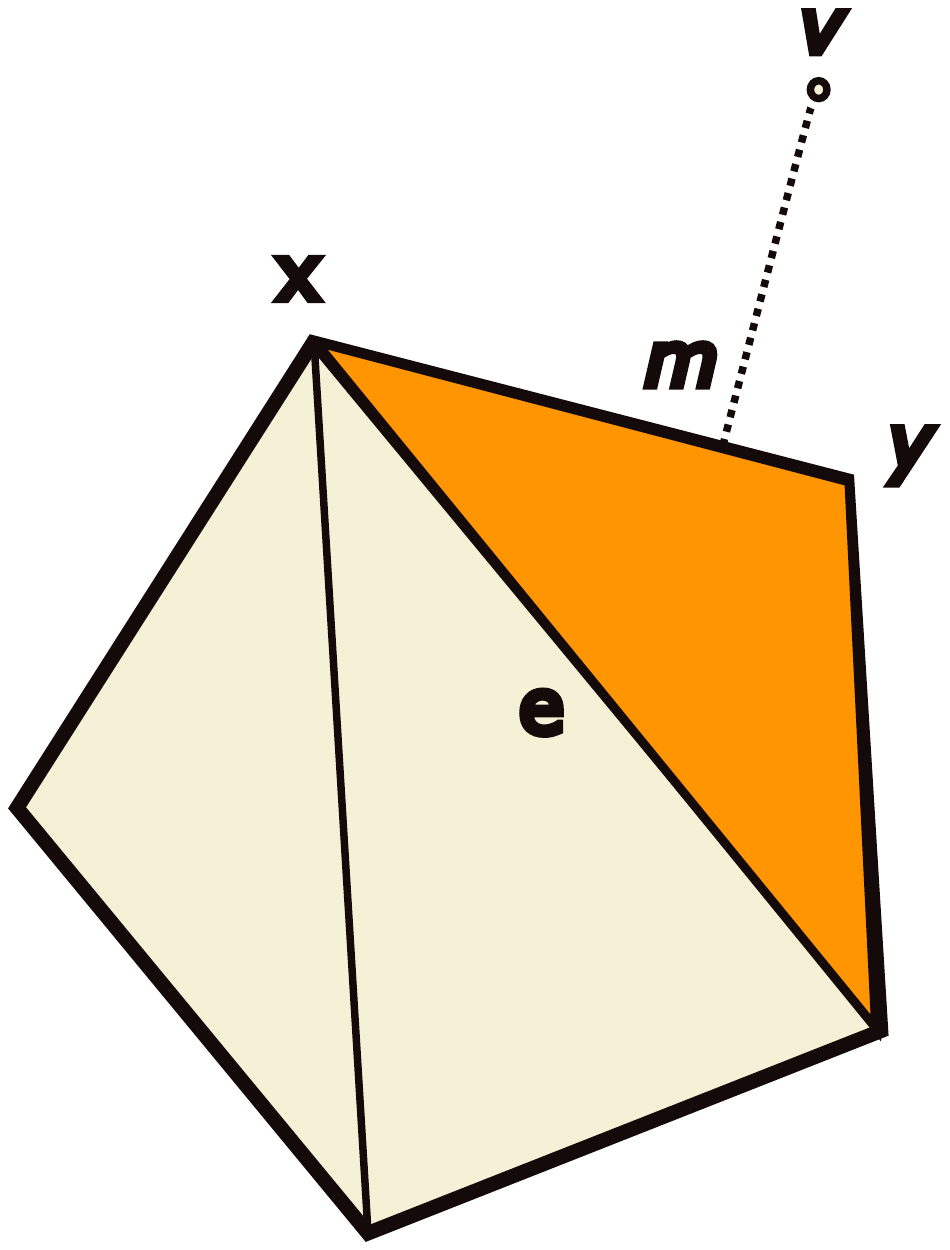}
 	\caption{\footnotesize Matching an edge $e$ with respect to the distance $f$ from some point $v$, with unique minimum $m$ on $|\St (e,C)|$. Then $\Theta_f(e)=e$ or 
 	$\Theta_f(e)=e\ast y$, according to whether $x$ is closer to $v$ than $y$ or not}
	\label{fig:matching}
\end{figure}

Let $\sigma$ be a face of $C$. On the star of $\sigma$, the function $f$ has a unique minimum $m$. We denote by $\mu (\sigma)$ the inclusion-minimal face among all the faces of $\St(\sigma,C)$ that contain the point $m$. Consider the  {\em pointer map} $y_f: C \rightarrow C$,  given by:  
\[y_f (\sigma) \:= \ y, \]
where $y$ is the unique vertex of $\mu (\sigma)$ on which $f$ attains its minimum value.

The matching $\Theta_f : C \longrightarrow C$ associated with the function $f$ is defined recursively on the dimension. Set first $\Theta(\emptyset) := \emptyset$. Let $\sigma$ be a face of $C$.  If for all faces $\tau \subsetneq \sigma$ one has $\Theta_f (\tau) \ne \sigma$, we define 
\[ \Theta_f(\sigma) \ := \ y_f (\sigma) \ast \sigma.\]
Note that every face of $C$ is either in the image of $\Theta_f$, or in its domain, or both. In the latter case, we have $\Theta_f (\sigma) = \sigma$. 

\begin{definition}[Gradient matching]
Let $\Theta : C \longrightarrow C$ be (the partial function associated to) a matching on the faces of a simplicial complex $C$. We say that $\Theta$ is a \emph{gradient matching} if 
$\Theta 
 =  
\Theta_f$
for some star-minimal function $f : |C| \rightarrow \mathbb{R}$.
\end{definition}

Next we recall the following result from \cite{ABpart1}, showing that all gradient matchings are  Morse matchings.
  
\begin{proposition}\label{thm:formula}
Let $C$ be a simplicial complex. Let $f:|C|\rightarrow \mathbb{R}$ be any star-minimal continuous function on the underlying space of~$C$. Then the induced gradient matching $\Theta_f$ is a Morse matching. Moreover, the map 
\[\sigma \mapsto (y_f(\sigma),\sigma)\]
yields a bijection between the set $\mathfrak{C}$ of the (nonempty) critical faces and the set \[\mathfrak{P} := \left \{
(v,\tau) \; | \; v \subset \tau, \,\ y_f(\tau)= v \textrm{ and } y_f(\tau- v)\ne v \right \}.\] 
In particular, the complex $C$ admits a discrete Morse function with $c_i$ critical $i$-simplices, where  
\[ c_i = \# \, \left\{
(v,\tau) \; | \; 
v \subset \tau, \, \dim \tau =i, \, y_f(\tau)= v  \textrm{ and } 
y_f(\tau- v )\ne v \right\}.
\]
\end{proposition}

\subsection{Arborescence and subdivisions of convex polytopes}

We will need later the following result from \cite{ABpart2}:

\begin{proposition}\label{thm:hudson}
Let $C, C'$ be polytopal complexes such that $C' \subset C$ and $C\searrow C'$. Let $D$ denote any subdivision of $C$, and define $D':=\RS(D,C')$. Then, $\sd D   \searrow   \sd D'$.
\end{proposition}

 The following lemma from \cite{ABpart1} seems well-known
we will include a proof for the sake of completeness:

\begin{lemma}\label{lem:BruMani} Let $\mu$ be any face of a $d$-dimensional polytope $P$. 
For any polytope $P$ and for any face $\mu$,  $P$ collapses onto 
$\St(\mu, \partial P)$, which is collapsible.
\end{lemma}

\begin{proof}
There exists a shelling of $\partial P$ starting with $\St(\sigma, \partial P)$ according to \cite[Cor.~8.13]{Z}.  
If $\tau$ is the last facet of such a shelling, then $P$ 
collapses onto $\partial P - \tau$ and the later collapses onto $\St(\sigma, \partial P)$ by using the shelling. 
\end{proof}

\begin{proposition}\label{thm:liccon}
Let $C$ denote any subdivision of a convex $d$-polytope. Then $\sd C$ is collapsible.
\end{proposition}

\begin{proof}
By Lemma~\ref{lem:BruMani}, the polytope $P$ is collapsible. Now for any subdivision $C$ of $P$, the facets of $C$ are all convex polytopes (not necessarily simplicial). Hence $\sd C$ is collapsible by Theorem~\ref{thm:hudson}.
\end{proof}

\section{Proofs of the main theorems} \label{sec:2}

\subsection{Proof of Theorem  \ref{mainthm:CATcollapsible}} \label{subsec:2}
Our approach consists in applying the ideas of Section~\ref{subsec:1} to the case where $C$ is a $\CAT(0)$ complex, and the function $f : C \rightarrow \mathbb{R}$ is the distance from a given vertex $v$ of $C$. Because of the $\CAT(0)$ property, this function is star-minimal and hence it induces a gradient matching on the complex. 

Recall that {\em ridges} are facets of facets of a simplicial complex. 

\begin{theorem}\label{thm:dischadamard}
Let $C$ be a  $\CAT(0)$ intrinsic locally finite simplicial complex such that
$\St (\sigma,C)$ in $C$ is convex for each  ridge $\sigma$. 
Then $C$ is arborescent.
\end{theorem}

\begin{proof}
In \cite{ABpart1} it was already proved that $\St (\sigma,C)$ in $C$ is convex for cell $\sigma$. 
If all ridges have convex stars, then for every vertex $v$ the closest-point projection to $\St(v,C)$ is a well-defined locally non-expansive map and hence globally non-expansive, as closed connected locally convex subsets are convex \cite{BuW}. 
If $x$ and $y$ are two points in $\St(v,C)$ for which the geodesic $\gamma$ from $x$ to $y$ leaves $\St(v,C)$. The projection $\gamma'$ of $\gamma$ on $\St(v,C)$ is a curve joining $x$ and $y$ which is contained in $\St(v,C)$ and  
is not longer than $\gamma$. This contradicts the uniqueness of  geodesic segments in $\CAT(0)$ spaces.

Fix $x$ in $C$ and let $\mathrm{d} : |C| \longmapsto \mathbb{R}$ be the distance from $x$ in $|C|$.  
By possibly slightly  moving $x$ there exists an unique vertex  $w$ of $C$ that minimizes $\mathrm{d}$. 
Note that $\mathrm{d}$ is a function that has a unique local minimum on the star of each face by Lemma~\ref{lem:unique}.
Let us perform on the face poset of $C$ the Morse matching constructed in Proposition~\ref{thm:formula}. The vertex $w$ will be mapped onto itself and for each vertex $u \neq w$ we have $y_{\mathrm{d}}(u) \ne u$. So  every vertex is matched with an edge, apart from $w$, which is the only critical vertex.

Assume that there is a critical face $\tau$ of dimension  at least $1$. Consider 
$v:=y_{\mathrm{d}}(\tau)$. By Lemma~\ref{lem:unique},
the restriction of the  function~${\mathrm{d}}$ to  $\St (\tau,C)$  attains its minimal value in the relative interior of a face $\sigma_v \in \St (\tau,C)$ that contains $v$.
Let $\delta$ be any face of $\St (\tau- v,C)$ containing $\sigma_v$, so that  $\delta$ contains $v$. 
Therefore $\St (\tau,C)$ and $\St (\tau- v,C)$ coincide in a neighborhood of $v$.
By Lemma~\ref{lem:unique}, the points where $\mathrm{d}$ attains its minima on   
$\St (\tau- v,C)$ and $\St(\tau,C)$ respectively, coincide. 
This implies that $y_{\mathrm{d}}(\tau- v)=y_{\mathrm{d}}(\tau)= v$, and hence:  
\[(v,\tau) \notin \{(v,\tau): v \subset \tau, \, y_{\mathrm{d}}(\tau)=~v \textrm{ and } y_{\mathrm{d}}(\tau- v) \ne v\}.\] 
By Proposition~\ref{thm:formula} $\tau$ is not critical, which contradicts our assumption.

We then infer from Theorem \ref{lem:equivalenceinf} that $C$ is arborescent. 
\end{proof}

Theorem~\ref{thm:dischadamard} can be extended even to polytopal complexes that are not simplicial.  

\begin{theorem}\label{thm:dischadamardpoly}
Let $C$ be any  $\CAT(0)$ intrinsic  locally finite polytopal complex such that
 for each face $\sigma$ in $C$, the underlying space of $\St (\sigma,C)$ in $C$ is convex. 
Then $C$ is arborescent. 
\end{theorem}

\begin{proof}
We will prove that every polytopal complex which admits a proper function $f:|C|\rightarrow \R$ that takes a unique local minimum on each vertex star is  arborescent.

We use  induction. Let $\sigma$ be a facet of $C$ maximizing $\min_{\sigma} f$, and let $\mu$ denote the strict face of $\sigma$ that minimizes  $f$. Let $F \subset \St (\mu,C)$ be the subcomplex induced by the facets of $C$ that attain their minimum at $\mu$. 
By Lemma~\ref{lem:BruMani}, we can collapse each facet $P$ of $F$ to $\St(\mu, \partial P)$. Hence, we can collapse $F$ to 
\[ \bigcup_{P \in F} \St(\mu, \partial P)  = \St (\mu,C) \: \cap \: \bigcup_{P \in F} \partial P, \]
where $P$ ranges over the facets of $F$. In particular, we can collapse $C$ to $C'=C-F$.

\begin{figure}[ht]
	\centering
	\vskip-1mm
  \includegraphics[width=.35\linewidth]{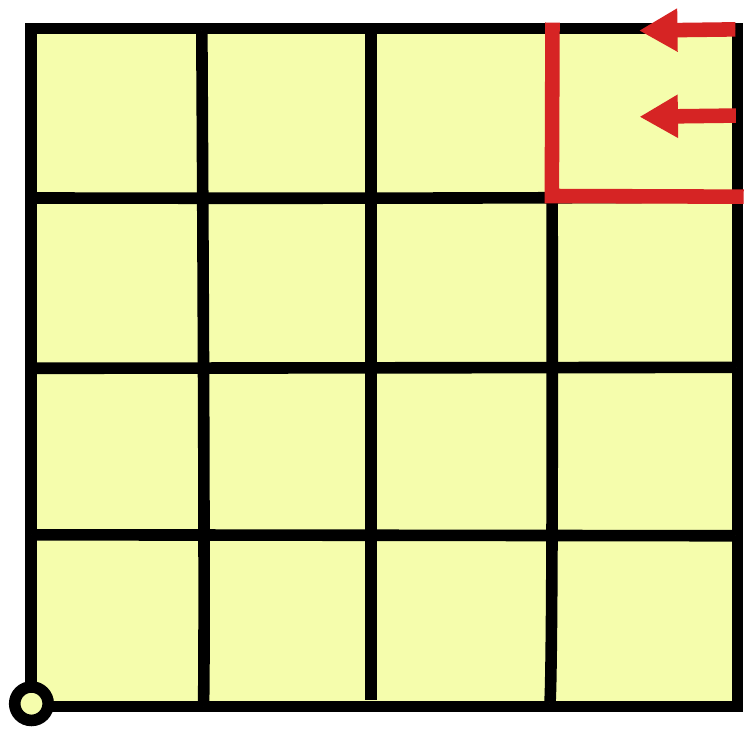}
 	\caption{\footnotesize Collapsing a $\CAT(0)$ cubical complex to the bottom left vertex $v$, where $\sigma$ is the top right square and $F$ is the red complex}
	\label{fig:grid}
\end{figure}

It remains to show that the restriction $f_{|C'}$ of $f$ to $C'$
attains a unique local minimum on each star. Assume the contrary, namely that for some vertex $w$ of $C$ the function $f_{|C'}$ has two local minima on $\St(w,C')$. Then  $w$ is in $F$. Let $x$ be the absolute minimum of $f$ restricted to $\St (w,C')$ and $y$ be the other local minimum. Let $P$ be a facet of $C-F$ containing $x$. When restricted to $C$, the function $f$ attains a unique local minimum on the star of every face. Therefore, the point $y$ must lie in $F$. In particular, the facet $P$ must contain the minimum of $f$ on $F$. But $y$ is not that local minimum since $P$ is not in $F$, so $f$ takes two local minima on $P$, contradicting the assumption on $f$.
\end{proof}

\subsection{Proof of Corollary \ref{mainthm:CATcube}}
A simplex is called \ \emph{non-obtuse} if the dihedral angle between any two facets is smaller than or equal to $\frac{\pi}{2}$. In any non-obtuse simplex, all faces are themselves non-obtuse simplices.  A simplex is called \emph{equilateral} or \emph{regular} if all edges have the same length. Equilateral simplices are obviously non-obtuse. By subdividing a cubical grid, one can obtain non-obtuse triangulations of $\mathbb{R}^d$ and of the $d$-cube for any $d$.
The next, straightforward lemma characterizes these notions in terms of orthogonal projections.

\begin{lemma}
A $d$-simplex $\Delta$ is \ non-obtuse if and only if, for each facet $F$ of $\partial \Delta$, the closest-point projection $\pi_{\Delta - F}\, \mathrm{span} F$ of the vertex $\Delta - F$ to the affine span of $F$ intersects $F$. 
\end{lemma}

Let $C$ be an intrinsic simplicial complex in which every face of $C$ is isometric to a regular euclidean simplex. If such $C$ is $\CAT(0)$, we say that it is \emph{$\CAT(0)$ with the equilateral flat metric}. More generally, a $\CAT(k)$ intrinsic simplicial complex $C$ is  \emph{$\CAT(k)$ with a non-obtuse metric} if every face of $C$ is non-obtuse.

We first show that
\begin{proposition} \label{cor:CrowleyA}
Let $C$ be an intrinsic locally finite simplicial complex. Suppose that $C$ is $\CAT(0)$ with a non-obtuse metric. 
Then $C$ is arborescent.
\end{proposition}

\begin{proof}
In non-obtuse triangulations, the star of every ridge is convex. In fact, let $\Sigma,\Sigma'$ be two facets containing a common ridge $R$. Observe that $\Sigma \cup \Sigma'$ is locally convex because both $\Sigma$ and $\Sigma'$ are convex and their union is locally convex in a neighborhood of $R$. Since $C$ is $\CAT(0)$, the convexity follows as in the first paragraph of the proof of Theorem \ref{thm:dischadamard}.
\end{proof}

\begin{cor} \label{cor:CrowleyB}
Every intrinsic locally finite  simplicial complex that is $\CAT(0)$ with the equilateral flat metric is arborescent.
\end{cor}

Further Theorem~\ref{thm:dischadamardpoly} holds for \emph{$\CAT(0)$ cube complexes}, which are complexes of regular unit cubes glued together to yield a $\CAT(0)$ metric space and yields: 

\begin{cor} \label{cor:CubeComplexes}
Every  $\CAT(0)$  locally finite cube complex is arborescent.
\end{cor}

\subsection{Proof of Corollary \ref{mainthm:CAT(0)man}}
We proved in \cite{AF} the following relation between $\CAT(0)$ complexes and arborescent complexes:

\begin{proposition}\label{thm:collcat}
 Let $C$ be any locally finite arborescent simplicial complex. Then there exists a $\CAT(0)$ cube complex $C'$ that is PL equivalent to $C$.
\end{proposition}

By Theorem~\ref{thm:dischadamardpoly}, the converse is also true, namely, if some $\CAT(0)$ cube complex $C'$ is PL homeomorphic to $C$, then $C'$ is arborescent. 
Therefore a simplicial complex admits an arborescent triangulation if and only if it is PL homeomorphic to a $\CAT(0)$ cubical complex.

\section{Barycentric subdivisions}

\subsection{Arborescence after subdivision}

Let $(S, \prec)$ be a poset and $S \subset T$. An \emph{extension} of $\prec$ to $T$ is a partial order $\widetilde{\prec}$ that coincides with $\prec$ when restricted to elements of $S$.

\begin{definition}[Derived order]\label{def:extord}
Let $C$ be a polytopal complex. Let $S$ denote a subset of mutually disjoint faces of $C$ and 
$\prec$ a partial order on $S$.   Let $\sigma$ be a face of $C$ and $\tau\subsetneq \sigma$ be a strict face of $\sigma$. 
\begin{enumerate}
\item If $\tau$ is the minimal face of $\sigma$ under $\prec$, we set $\tau\,\widetilde{\prec}\, \sigma$.
\item If $\tau$ is any other face of $\sigma$, we set $\sigma\, \widetilde{\prec}\, \tau$.
\end{enumerate}
The transitive closure of the relation $\widetilde{\prec}$ gives an irreflexive partial order on the faces of $C$, which is an extension of $\prec$. 
The correspondence between faces of $C$ and vertices of $\sd  C$, provides an irreflexive partial order on $\F_0(\sd  C)$. A total order that extends the later  
order is a \emph{derived order} of $\F_0(\sd  C)$ induced by $\prec$.
\end{definition}

\begin{definition}\label{def:slk}
Let $C$ be a simplicial complex endowed with a $\CAT(\kappa)$ polyhedral metric. 
Let $v$ be a vertex of $C$ and $F$ be a closed convex subset of $C$, $v\in \partial F$. 
Given a unit tangent vector $\nu$ at $p$ pointing inward $F$ we consider the totally geodesic support hyperplane $H^{\nu}$ at $p$ for $\partial F$ which is orthogonal to $\nu$. Let $\overline{H}_+$ be the closed 
connected component of $C\setminus H^{\nu}$  containing $F$ locally. 
The \emph{descending link} $\LLk (v,C)$ of~$C$ at $v$ with respect to the direction $\nu$ is the restriction $\RS(\Lk(v,C), \intx \TT_v^1 \overline{H}_+^{\nu})$ of $\Lk(v,C)$ to $\intx \TT_v^1 \overline{H}_+^{\nu}$, where 
\[\TT_v^1 \overline{H}_+^{\nu}= \{x\in \TT_v^1 \overline{H}_+; \measuredangle(x,\nu) < \frac{\pi}{2}\}.\]
\end{definition}

\subsection{Collapsibility of compact CAT(0) polyhedral complexes}
A key criterion to be used in the sequel is the following:

\begin{lemma}[Gromov--Alexandrov lemma; cf.\ \cite{BH}, Thm. II.5.4]
Let $\Sigma$ be a locally finite simplicial complex endowed with a complete metric  for which simplices are isometric to 
constant curvature $\kappa$ simplices (and only finitely many isometry classes occur). Then $\Sigma$ is a $\CAT(\kappa)$ space if and only if the link of every vertex of 
$\Sigma$ is  $\CAT(\kappa)$ with the induced spherical metric and moreover it contains no isometrically embedded  
circles of length less than $2\pi/\sqrt{\kappa}$ when $\kappa >0$, or equivalently $\Sigma$ is $\pi/\sqrt{\kappa}$ - uniquely geodesic.
\end{lemma}

Note also that a $\CAT(1)$-space is $\pi$-uniquely geodesic (see \cite{BH}, Prop. II.1.4). This implies that 
the link of every vertex of it is also $\pi$-uniquely geodesic, when endowed with the spherical metric and hence any 
link subcomplex of diameter smaller than $\pi$ is also a $\CAT(1)$-space. 
We begin with the following result for the compact case: 

\begin{theorem}\label{thmconvexcollapse}
Suppose that $C$ is a compact convex polytopal $d$-subcomplex of a polytopal complex $X$ and  
$F$ is a closed convex subset of $X$ in general position with respect to $C$,  such that: 
\begin{enumerate} 
\item either $X$ is endowed with a polyhedral $\CAT(0)$-metric;
\item or $X$ is endowed with a polyhedral $\CAT(1)$ metric and the diameter of $ F $ is smaller than ${\pi}$.  
\end{enumerate}
If $\partial C\cap  F =\emptyset$, then $N(\RS(C, F ),C)$ is collapsible.
 If $\partial C\cap  F $ is nonempty, and $C$ does not lie in $ F $, then $N(\RS(C, F ),C)$ collapses to the subcomplex $N(\RS(\partial C, F ),\partial C)$.

\end{theorem}

\begin{proof}[\textbf{Proof of Theorem~\ref{thmconvexcollapse}}]
We use  induction on the dimension. The claim is true for 0-dimensional polytopal complexes $C$. 
Assume, from now on, that $d>0$ and the claim holds for dimensions $\leq d-1$. Recall that $\mathrm{L}(\tau,C)$ denotes the faces of $C$ strictly containing a face $\tau$ of $C$. 
\smallskip

\noindent  Let $p$ be an interior point  of $ F $, and  $\mathrm{d}_p(y)$ denote the distance of a point $y \in X$ to $p$ with respect to the polyhedral $\CAT(\kappa)$ metric on $X$.

Let $\mathrm{M}(C,  F )$ denote the faces $\sigma$ of $\RS(C, F )$ for which the function ${\min}_{y\in \sigma} \mathrm{d}_p(y)$ attains its minimum in the relative interior of $\sigma$. We order the elements of $\mathrm{M}(C,  F )$ strictly by defining $\sigma\prec \sigma'$ whenever $\min_{y\in\sigma}\mathrm{d}_p(y)<\min_{y\in\sigma'}\mathrm{d}_p(y)$.

There is an associated derived order on the vertices of $\sd C$, which we restrict to the vertices of $N(\RS(C, F ),C)$. Let $v_0, v_1, v_2, \, \cdots, v_n$ denote the vertices of $N(\RS(C, F ),C)$ labeled according to the latter order, starting with the minimal element $v_0$. Let $C_i$ denote the complex $N(\RS(C, F ),C)-\{v_i, v_{i+1}, \, \cdots, v_{n}\},$ and define 
\[\Sigma_i:=C_i\cup N(\RS(\partial C, F ),\partial C).\] We will prove that $\Sigma_i\searrow \Sigma_{i-1}$ for all $i$, $1\le i\le n$, proving our claims. There are four cases to consider here.

\begin{compactenum}[(1)]
\item $v_i$ is in the interior of $\sd C$ and corresponds to an element of $\mathrm{M}(C,  F )$.
\item $v_i$ is in the interior of $\sd C$ and corresponds to a face of $C$ not in $\mathrm{M}(C,  F )$.
\item $v_i$ is in the boundary of $\sd C$ and corresponds to an element of $\mathrm{M}(C,  F )$.
\item $v_i$ is in the boundary of $\sd C$ and corresponds to a face of $C$ not in $\mathrm{M}(C,  F )$.
\end{compactenum}
Let us denote by $\tau$ the face of $C$ corresponding to $v_i$ in $\sd C$, and let $m$ denote the unique point realizing the 
mimimum ${\min}_{y\in \tau} \mathrm{d}_p(y)$. Consider further the metric ball $B_m=\{y\in X; \mathrm{d}_p(y)\le\mathrm{d}_p(m)\}$. 

\smallskip

Observe that  $\RN^1_{(m,\tau)} X$ was identified  with the link $\Lk_{m}(\tau, X)$. By hypothesis 
$\RN^1_{(m,\tau)} X$ is a $\CAT(1)$-space. 
Moreover, $\tau$ is a totally geodesic submanifold of $X$ and $B_m$ a convex subset of $X$. We saw how to define then 
 $N^1_{(m,\tau)} B_m$, which is a subset of $\RN^1_{(m,\tau)} X$. From \cite{AF} $B_m$ is a polyhedron, under our assumptions. 
  Consider the geodesic issued joining $m$ to $p$. 
Its tangent vector at $m$ is a well-defined vector $\nu\in N^1_{(m,\tau)} B_m$, since $m$ realizes the 
unique minimum of the distance function $d_p$ to  the convex cell $\tau$. Therefore $N^1_{(m,\tau)} B_m$ is contained within the set 
of unit vectors making an angle with $\nu$ smaller than $\frac{\pi}{2}$, so that its diameter is at most ${\pi}$. It follows that 
 $N^1_{(m,\tau)} B_m$ is a convex subset of the $\CAT(1)$-space $\RN^1_{(m,\tau)} X$. 
As the distance function $d_p$ is strictly convex, the subcomplex $\RS(\Lk_m(\tau,C),  \RN^1_{(m,\tau)} B_m)$ has actually a diameter strictly smaller than $\pi$. 

\smallskip
\noindent \emph{Case $(1)$}:  The complex $\Lk(v_i,\varSigma_i)$ is combinatorially equivalent to $N(\mathrm{LLk}_m(\tau,C),\Lk_m(\tau,C))$, where \[\mathrm{LLk}_m(\tau,C):= \RS(\Lk_m(\tau,C),  \RN^1_{(m,\tau)} B_m)\] is the restriction of $\Lk_m(\tau,C)$ to the subspace $ \RN^1_{(m,\tau)} B_m$ of $\RN^1_{(m,\tau)} X$.  As $p$ was chosen to be generic, $\RN^1_{(m,\tau)} B_m$ is in general position with respect to $\Lk_m(\tau,C)$.

Hence, by  the induction assumption
the complex \[N(\mathrm{LLk}_m(\tau,C),\Lk_m(\tau,C))\cong \Lk (v_i,\varSigma_i)\] is collapsible. Consequently,  Lemma~\ref{lem:cecoll} proves $\varSigma_i\searrow \varSigma_{i-1}=\varSigma_i-v_i.$

\smallskip
\noindent \emph{Case $(2)$}: If $\tau$ is not an element of $\mathrm{M}(C,  F )$, let $\sigma$ denote the face of $\tau$ containing $m$ in its relative interior. Then, $\Lk(v_i, \varSigma_i)=\Lk(v_i, C_i)$ is combinatorially equivalent to the order complex of the union $\mathrm{L}(\tau,C)\cup \sigma$, whose elements are ordered by inclusion. Since $\sigma$ is a unique global minimum of the poset, the complex $\Lk(v_i, \varSigma_i)$ is a cone, and in fact combinatorially equivalent to a cone over base $\sd \Lk(\tau, C)$. But all cones are collapsible (Lemma~\ref{lem:ccoll}), so $\Lk(v_i, \varSigma_i)$ is collapsible. Consequently, Lemma~\ref{lem:cecoll} gives $\varSigma_i\searrow \varSigma_{i-1}=\varSigma_i-v_i$.

\smallskip
\noindent \emph{Case $(3)$}: This time, $v_i$ is in the boundary of $\sd C$. As in case {(1)}, $\Lk (v_i,C_i)$ is combinatorially equivalent to the complex 
\[N(\mathrm{LLk}_m(\tau,C),\Lk_m(\tau,C)),\  \  \mathrm{LLk}_m(\tau,C):= \RS(\Lk_m(\tau,C),  \RN^1_{(m,\tau)} B_m)\] 
in the $\CAT(1)$-space  $\RN^1_{(m,\tau)} X$. 
As $\tau$ is not the face of $C$ that minimizes $\mathrm{d}_p(y)$ since $v_i\neq v_n$, so that $\RN^1_{(m,\tau)} B_m\cap \RN^1_{(m,\tau)} \partial C$ is nonempty. Since furthermore $\RN^1_{(m,\tau)} B_m$ is a $\CAT(1)$ space of diameter smaller than $\pi$ 
it is also convex. Being in general position with respect to the complex $\Lk_m(\tau,C)$ in $\RN^1_{(m,\tau)} X$, the inductive assumption  applies.  Thus the complex $N(\mathrm{LLk}_m(\tau,C),\Lk_m(\tau,C))$ collapses to 
\[N(\mathrm{LLk}_m(\tau,\partial C),\Lk_m(\tau,\partial C))\cong \Lk (v_i,C'_i),\ \  C'_i:=C_{i-1} \cup (C_i\cap N(\RS(\partial C, F ),\partial C)).\]
Consequently, Lemma~\ref{lem:cecoll} proves that $C_i$ collapses to $C'_i$. Since \[ \varSigma_{i-1}\cap C_i =(C_{i-1}\cup N(\RS(\partial C, F ),\partial C))\cap C_i=  C_{i-1} \cup (C_i\cap N(\RS(\partial C, F ),\partial C))= C'_i\]
Lemma~\ref{lem:uc}, applied to the union $\varSigma_i=C_i\cup \varSigma_{i-1}$ of complexes $C_i$ and $\varSigma_{i-1}$ gives that $\varSigma_i$ collapses onto~$\varSigma_{i-1}.$

\smallskip
\noindent \emph{Case $(4)$}: As observed in case {(2)}, the complex $\Lk(v_i,C_i)$ is  combinatorially equivalent to a cone over base $\sd \Lk(\tau, C)$, which collapses to the cone over the subcomplex $ \sd  \Lk(\tau, \partial C)$ by Lemma~\ref{lem:ccoll}. Thus, the complex $C_i$ collapses to $C'_i:=C_{i-1}\cup(C_i \cap N(\RS(\partial C, F ),\partial C))$ by Lemma~\ref{lem:cecoll}. Now, we have $\varSigma_{i-1}\cap C_i=C'_i$ as in case {(3)}, so that $\varSigma_i$ collapses onto $\varSigma_{i-1}$ by Lemma~\ref{lem:uc}.

\smallskip

This finishes the proof of Theorem~\ref{thmconvexcollapse}.

\end{proof}

\subsection{Proof of Theorem \ref{mainthm:CATpolyhedral}}
Our aim is to show that: 
\begin{theorem}\label{arborescent}
A locally finite simplicial complex which is CAT(0) for a polyhedral metric 
has its first barycentric subdivision arborescent.  
\end{theorem}

\begin{proof}
We follow the lines of the proof of Theorem~\ref{thmconvexcollapse}.
We use  induction on the dimension. The claim is true for 0-dimensional polytopal complexes $C$. 
Assume, from now on, that $d>0$ and the claim holds for dimensions $\leq d-1$. 
\smallskip

 Let $p$ be an interior point  of $ X $, and  $\mathrm{d}_p(y)$ denote the distance of a point $y \in X$ to $p$ with respect to the polyhedral $\CAT(\kappa)$ metric on $X$.

Let $\mathrm{M}(X )$ denote the faces $\sigma$ of $X$ for which the function ${\min}_{y\in \sigma} \mathrm{d}_p(y)$ attains its minimum in the relative interior of $\sigma$. We order the elements of $\mathrm{M}(X)$ strictly by defining $\sigma\prec \sigma'$ whenever $\min_{y\in\sigma}\mathrm{d}_p(y)<\min_{y\in\sigma'}\mathrm{d}_p(y)$.

There is an associated derived order on the vertices of $\sd X$. Let $v_0, v_1, v_2, \, \ldots, v_n,\ldots$ denote the vertices of $\sd X$ labeled according to the latter order, starting with the minimal element $v_0$. Let $\Sigma_i$ denote the complex 
$X-\{v_j,  j\geq i\}$. We claim that $\Sigma_i\searrow \Sigma_{i-1}$ for all $i\geq 1$. The proof given above in the compact case works here without essential changes. 
\end{proof}

\section{Arborescent manifolds}
\subsection{Arborescent PL triangulations} \label{subsec:3}
Whitehead proved in \cite{Whitehead}: 

\begin{proposition} Let $M$ be a (compact) manifold with boundary. If some PL triangulation of~$M$ is collapsible, then $M$ is a ball.
\end{proposition}

The following  is a extension of Whitehead's theorem to  the case of open manifolds, 
which was independently obtained by Guilbault and Gulbrandsen (see \cite{GuiGul}) with a different proof: 

\begin{proposition}\label{openwhitehead} 
Let $M$ be an open manifold. If some PL triangulation of $M$ is arborescent, then $M$ is PL homeomorphic to the Euclidean space.
\end{proposition}
\begin{proof}
By (\cite{AF}, Thm. 4) there exists a locally finite 
CAT(0) cubical complex $K$ which is PL homeomorphic to $M$. 
According to Stone's theorem \cite{Stone} a simply connected manifold M endowed with a PL triangulation for which the associated polyhedral piecewise flat  metric is CAT(0) should be PL homeomorphic to the Euclidean space. 
It follows that $M$ is PL homeomorphic to the Euclidean space. 
\end{proof}
As an immediate consequence the Whitehead 3-manifold (see \cite{Whitehead1935}) and its generalizations in higher dimensions \cite{FG} do not admit $\CAT(0)$ polyhedral metrics. This is particularly interesting when the dimension is 4. In fact open contractible manifold is  triangulable and by Perelman any triangulation of a 4-manifold should be PL. 
By \cite{Lytchak} an open $\CAT(0)$-manifold is homeomorphic to $\R^4$. However, if the 
$\CAT(0)$-metric is polyhedral, then the manifold is PL homeomorphic to $\R^4$, namely 
it cannot be an exotic $\R^4$.

\subsection{Polyhedral CAT(0) Manifolds}
We now focus on non-PL triangulations. Let us introduce a convenient notation:
\begin{compactitem}[$-$]
\item {$\CAT_d^\Box$} denotes the class of open $d$-manifolds PL homeomorphic to $\CAT(0)$ cube complexes;
\item {$\CAT_d$} denotes the class of open $d$-manifolds admitting a complete polyhedral $\CAT(0)$ metric with finitely many shapes;
\item {$\mathrm{A}_d$} denotes the open $d$-manifolds that admit an arborescent triangulation; 
\item {$\mathrm{PC}_d$} denotes all open contractible $d$-manifolds which are pseudo-collarable, have strongly semistable fundamental group at infinity and have vanishing Chapman-Siebenmann obstruction (see \cite{AF}).
\end{compactitem}

\begin{proposition}\label{thm:ccc}
For each $d \ge 5$ one has
\[\CAT_d^\Box=\mathrm{A}_d \subseteq \mathrm{PC}_d,\]
whereas for $d\leq 4$ one has
\[\CAT_d^\Box=\mathrm{A}_d= \{\R^d\}.\]
\end{proposition}

\begin{proof}
By Theorem~\ref{thm:dischadamardpoly}, every locally finite $\CAT(0)$ cube complex is arborescent, and so is its first derived subdivision.  \cite{Welker}.  Hence  locally finite $\CAT(0)$ cube complexes admit arborescent triangulations. 
Conversely, locally finite arborescent complexes  are PL homeomorphic to CAT(0) cubical complexes (see \cite{AF}, Thm. 4) and  thus
\[\CAT_d^\Box=\mathrm{A}_d  \: \textrm{ for all } d.\] 
We showed in (\cite{AF}), Thm.1) that 
\[\CAT_d\subseteq \mathrm{PC}_d \: \textrm{ for all } d\geq 5.\] 

When $d \le 4$, every triangulation of a $d$-manifold is PL. This is non-trivial,  for $d=4$, this statement relies on the Poincar\'e Conjecture, a.k.a. Perelman theorem. By  the Whitehead type theorem \ref{openwhitehead}, this implies that $\mathrm{CAT}_d= \{\R^d\}$. 
All open CAT(0) $3$-manifolds are homeomorphic to the Euclidean 3-space \cite{Rolfsen}, so 
\[\{\R^d\} =  \CAT_d = \mathrm{A}_d =\mathrm{PC}_d \: \textrm{ for } d\le 3.\] 
\[\{\R^4\} = \CAT_4^\Box=A_4\subseteq \mathrm{PC}_4.\]
This improves the result of \cite{Lytchak} in a particular case.  
\end{proof}

\begin{rem}
An optimistic conjecture is that $\mathrm{A}_d= \mathrm{PC}_d$ for $d\geq 5$. 
We proved in \cite{AF} that $\mathrm{PC}_d\subseteq \mathrm{A}^{TOP}_d$, for $d\geq 5$, where 
{$\mathrm{A}^{TOP}_d$} we denote the open $d$-manifolds that are topologically arborescent. 
For $d\geq 5$  all we need to show is that $\mathrm{PC}_d \subseteq \CAT_d^\Box$, namely, that every open contractible $d$-manifold, $d\ge 5$,
pseudo-collarable, having strongly semistable fundamental group at infinity and vanishing Chapman-Siebenmann obstruction 
admits a $\CAT(0)$ cube structure. 
\end{rem}


\end{document}